\documentclass[12pt]{article}
\usepackage[utf8]{inputenc}
\usepackage{amsmath, amssymb, amsthm, verbatim}
\usepackage[dvipsnames]{xcolor}
\usepackage{hyperref}
\usepackage{wrapfig}
\usepackage{tikz}
\usepackage{soul}
\usepackage{mathtools}
\usepackage[normalem]{ulem}
\usepackage{fullpage}
\usetikzlibrary{tikzmark,decorations.pathreplacing}
\tikzset{dot/.style={circle,draw,scale =.5 , fill = black}}

\hypersetup{hidelinks}

\newcommand{\Z}{\mathbb{Z}}

\newcommand{\sbar}{\overline{s}}
\newcommand{\esye}{\widetilde{s}}

\newcommand{\ang}[1]{\langle #1 \rangle}
\newcommand{\ceil}[1]{\left\lceil #1 \right\rceil}
\newcommand{\floor}[1]{\left\lfloor #1 \right\rfloor}

\newcommand{\abs}[1]{\left|#1\right|}

\theoremstyle{plain}

\newtheorem{thm}{Theorem}[section]
\newtheorem{prop}[thm]{Proposition}

\newtheorem{lemma}[thm]{Lemma}

\theoremstyle{definition}
\newtheorem{definition}[thm]{Definition}
\theoremstyle{remark}

\title{The maximum hook length of $d$-distinct simultaneous core partitions}
\author{Ethan Pesikoff\thanks{Yale University, {\tt ethan.pesikoff@yale.edu}.} \and Benjamin Przybocki\thanks{Stanford University, {\tt benprz@stanford.edu}.} \and Janabel Xia\thanks{MIT, {\tt janabel@mit.edu}.}}
\date{\today}

\begin{document}

\maketitle

\begin{abstract}
    We exactly determine the maximum possible hook length of $(s,t)$-core partitions with $d$-distinct parts when there are finitely many such partitions. Moreover, we provide an algorithm to construct a $d$-distinct $(s,t)$-core partition with this maximum possible hook length.
\end{abstract}

\section{Introduction} \label{sec-intro}

A \emph{partition} is a weakly decreasing tuple of positive integers $\lambda = (\lambda_1, \lambda_2, \dots, \lambda_n)$. The \emph{size} of $\lambda$ is $\lambda_1 + \lambda_2 + \dots + \lambda_n$. Partitions have been studied not only for their number-theoretic and combinatorial properties, but also for their applications to the representation theory of the symmetric group.

A partition can be visualized by its \emph{Young diagram}, which is a left-justified array of cells where row $i$ contains $\lambda_i$ cells for all $i \in [n]$. For each cell, we define its \emph{hook} to be all the cells on its right, all the cells below it, and itself. The \emph{hook length} of a cell is the number of cells in its hook. (See Figure~\ref{fig-young-diagram}.) A notion of interest in representation theory is that of an \emph{$s$-core partition}, a partition whose Young diagram contains no cells with hook length $s$ \cite[Chapter~2]{jk1981}. Throughout this paper, we simply refer to an $s$-core partition as an \emph{$s$-core}.

\begin{figure}[!ht]
    \centering
    \begin{tikzpicture}[scale=0.55]
        \fill [orange] (3, 2) rectangle (4, 3);
        \fill [orange] (3, 3) rectangle (4, 4);
        \fill [orange] (4, 3) rectangle (5, 4);
        \fill [orange] (5, 3) rectangle (6, 4);
        \fill [orange] (6, 3) rectangle (7, 4);
        \fill [orange] (7, 3) rectangle (8, 4);
    
        \draw (0,4) -- (8,4);
        \draw (0,3) -- (8,3);
        \draw (0,2) -- (6,2);
        \draw (0,1) -- (3,1);
        \draw (0,0) -- (1,0);
        
        \draw (0,0) -- (0,4);
        \draw (1,0) -- (1,4);
        \draw (2,1) -- (2,4);
        \draw (3,1) -- (3,4);
        \draw (4,2) -- (4,4);
        \draw (5,2) -- (5,4);
        \draw (6,2) -- (6,4);
        \draw (7,3) -- (7,4);
        \draw (8,3) -- (8,4);
        
        \node at (0.5, 0.5) {1};
        \node at (0.5, 1.5) {4};
        \node at (1.5, 1.5) {2};
        \node at (2.5, 1.5) {1};
        \node at (0.5, 2.5) {8};
        \node at (1.5, 2.5) {6};
        \node at (2.5, 2.5) {5};
        \node at (3.5, 2.5) {3};
        \node at (4.5, 2.5) {2};
        \node at (5.5, 2.5) {1};
        \node at (0.5, 3.5) {11};
        \node at (1.5, 3.5) {9};
        \node at (2.5, 3.5) {8};
        \node at (3.5, 3.5) {6};
        \node at (4.5, 3.5) {5};
        \node at (5.5, 3.5) {4};
        \node at (6.5, 3.5) {2};
        \node at (7.5, 3.5) {1};
    \end{tikzpicture}
    \caption{The Young diagram of $\lambda = (8, 6, 3, 1)$. The orange cells compose a hook, and the numerals indicate the hook length of each cell.}
    \label{fig-young-diagram}
\end{figure}

Anderson~\cite{anderson2002} generalized this notion to that of an \emph{$(s,t)$-core}, which contain no cells with hook length $s$ or $t$. (For example, we can see from Figure~\ref{fig-young-diagram} that $\lambda = (8,6,3,1)$ is a $(7,10)$-core.) In particular, she proved that there are $\binom{s+t}{s}/(s+t)$ such cores when $s$ and $t$ are coprime; otherwise, there are infinitely many. Anderson's result has inspired several research directions related to $(s,t)$-cores (see \cite{ahj2014,nath2017} and \cite[Section~4]{ckns2021} for three surveys on the subject).

One such direction has studied $(s,t)$-cores with distinct parts (see, e.g, \cite{straub2016, ns2017, yqjz2017, zz2017, bny2018, xiong2018}), in which $\lambda_i - \lambda_{i+1} \ge 1$ for all $i \in [n-1]$. We refer to such cores as \emph{distinct} $(s,t)$-cores. More generally, one can study \emph{$d$-distinct} $(s,t)$-cores \cite{sahin2018, kravitz2019, bs-ts2021}, in which $\lambda_i - \lambda_{i+1} \ge d$ for all $i \in [n-1]$. Kravitz~\cite[Lemma~2.4]{kravitz2019} proved that the number of $d$-distinct $(s,t)$-cores is finite if and only if $\gcd(s,t) \le d$, extending Anderson's result. Most work has focused on counting $d$-distinct $(s,t)$-cores, which has only been solved for a few choices of parameters. Similarly, closed-form expressions for the maximum size, maximum number of parts, and maximum possible hook length (also known as \emph{perimeter}) of $d$-distinct $(s,t)$-cores were only known for a few choices of parameters.

The purpose of this paper is to present a closed-form expression for the maximum possible hook length of $d$-distinct $(s,t)$-cores when there are finitely many such cores. Only loose bounds for general $s$ and $t$ were previously known. Our main theorem, proved in Section~\ref{sec-coprime}, handles the case when $s$ and $t$ are coprime.

\begin{thm} \label{thm-max-hook-coprime}
    Let $s,k,d \in \Z_{>0}$ with $s$ and $k$ coprime and $s \ge 2$. Then, the maximum possible hook length $H_d$ of an $(s,s+k)$-core with $d$-distinct parts is
    \[
        H_d(s,k) = \begin{cases}
            s-1 &\quad\text{if} \ k = 1 \ \text{or} \ k,s \leq d \\
            s+k-1 &\quad\text{if} \ 1 < k \le d < s \\
            B-2 &\quad\text{if} \ d < k \ \text{and} \ \sbar \esye \bmod k = 1 \\
            B-s-1 &\quad\text{if} \ 1 < \sbar \esye \bmod k \le d < k\\
            B + k - \sbar \esye -1 &\quad\text{if} \ d < \sbar \esye \bmod k < k-1 \\
            B-1 &\quad\text{if} \ d < \sbar \esye \bmod k = k - 1,
        \end{cases}
    \]
    where
    \begin{align*}
        B &= \floor{\frac{s-1}{k}} \left(k + s\esye\right) + s\left( \ceil{\frac{\sbar\esye-1}{k}} + \esye - 1\right)+\sbar, \\
        \sbar &= s \bmod k, \ \text{and} \\
        \esye &= \min \{\ell \cdot (\sbar)^{-1} \bmod k \mid -d \leq \ell \leq d, \ \ell\neq0\}.
    \end{align*}
\end{thm}

Note that we use $a \bmod b$ to denote the modulo operation (remainder of Euclidean division of $a$ by $b$) and $a \pmod b$ to denote $a$ as an element of $\Z/b\Z$.

Then, in Section~\ref{sec-non-coprime}, we extend our result to all $s$ and $t$ satisfying $\gcd(s,t) \le d$, which resolves the problem for all choices of parameters.

\begin{thm} \label{thm-max-hook}
    Let $s,k,d \in \Z_{>0}$ with $s$ and $k$ coprime and $s \ge 2$. Then, for all integers $b \ge 2$ and $0\leq c < b$, we have
    \[
        H_{bd+c}(bs,bk)=\begin{cases}
            b\left(H_{d}\left(s,k\right)+2\right)-1 &\quad\text{if} \ k=1 \ \text{and} \ d<s\\
            b\left(H_{d}\left(s,k\right)+1\right)-1 & \quad\text{if} \ k=1 \ \text{and} \ d\geq s \\
            b\left(H_{d}\left(s,k\right)+2\right)-1 &\begin{aligned}
                \quad\text{if}& \ d < k \ \text{and} \ (\sbar\esye\bmod k = 1 \\
                \quad &\text{or} \ d < \sbar \esye \bmod k = k-1)
            \end{aligned} \\
            b\left(H_{d}\left(s,k\right)+1\right)-1 &\begin{aligned}
                \quad\text{if}& \ k > 1 \ \text{and} \ (1 < \sbar \esye \bmod k \le d \\
                \quad &\text{or} \ (d < \sbar \esye \bmod k < k-1) \ \text{or} \ d\geq k).
            \end{aligned}
        \end{cases}
    \]
\end{thm}

\section{Background}

For a partition $\lambda = (\lambda_1, \lambda_2, \dots, \lambda_n)$, its \emph{$\beta$-set} is
\[
    \beta(\lambda) = \{\lambda_1 + n - 1, \lambda_2 + n - 2, \dots, \lambda_n\}.
\]
Equivalently, $\beta(\lambda)$ is the set of hook lengths of the cells in the first column of the Young diagram of $\lambda$. For example, we can see from Figure~\ref{fig-young-diagram} that $\beta(8,6,3,1) = \{11,8,4,1\}$. Hence, the maximum hook length of a given partition is the greatest element of its $\beta$-set. The function $\beta$ is a bijection from the set of partitions to the set of finite subsets of $\Z_{>0}$.

For our purposes, it's easier to work with $\beta$-sets rather than tuples of parts. This is because of the following characterization of $s$-cores.

\begin{prop}[{\cite[Lemma~2.7.13]{jk1981}}] \label{prop-s-core}
    A partition $\lambda$ is an $s$-core if and only if for all $x \in \beta(\lambda)$ with $x \ge s$, we have $x-s \in \beta(\lambda)$.
\end{prop}

We can also characterize $d$-distinct partitions in terms of their $\beta$-sets.

\begin{prop}[{\cite[Lemma~2.1]{sahin2018}}]
    A partition $\lambda$ is $d$-distinct if and only if for all $x,y \in \beta(\lambda)$ with $x \neq y$, we have $|x-y| > d$.
\end{prop}

Proposition~\ref{prop-s-core} motivates the definition of the following poset, which is implicitly used in \cite{anderson2002}.

\begin{definition}
    Let
    \[
        \mathcal{P}_{s, s+k} = \Z_{>0} \setminus \{x \in \Z_{>0} \mid x = as+b(s+k) \ \text{for some} \ a,b\in \Z_{\geq 0}\}.
    \]
    For $x,y \in \mathcal{P}_{s, s+k}$, let $x \lessdot_{\mathcal{P}_{s, s+k}} y$ if $y-x \in \{s,s+k\}$. Then, $<_{\mathcal{P}_{s, s+k}}$ is the transitive closure of $\lessdot_{\mathcal{P}_{s, s+k}}$.
\end{definition}

An \emph{order ideal} $\mathcal{X}$ is a subset of $\mathcal{P}_{s, s+k}$ such that if $x \in \mathcal{X}$ and $y <_{\mathcal{P}_{s, s+k}} x$, then $y \in \mathcal{X}$. We use $\ang{x}$ to denote the order ideal generated by $x \in \mathcal{P}_{s, s+k}$.

By Proposition~\ref{prop-s-core}, the $\beta$-sets of $(s,s+k)$-cores are exactly the order ideals of $\mathcal{P}_{s, s+k}$. For example, Figure~\ref{fig-p} illustrates the order ideal $\{11,8,4,1\} \subseteq \mathcal{P}_{7,10}$, which gives another way of seeing that $\lambda = (8,6,3,1)$ is a $(7,10)$-core.

\begin{figure}[!ht]
    \centering
    \begin{tikzpicture}[scale=0.8]
        \def\s{7}
        \def\t{10}
        \def\F{\s*\t-\s-\t}
        \foreach \a in {0, 1,...,\t} {
            \foreach \b in {0, 1,...,\s} {
                \pgfmathparse{int(\F-\a*\s-\b*\t)}
                \ifnum \pgfmathresult>0
                    \node[label={\pgfmathresult}] [dot] (\a/\b) at (\b-\a,-\a-\b) {};
                    \ifnum \a > 0
                        \pgfmathtruncatemacro{\prev}{\a-1}
                        \draw [thick] (\a/\b) -- (\prev/\b);
                    \fi
                    \ifnum \b > 0
                        \pgfmathtruncatemacro{\prev}{\b-1}
                        \draw [thick] (\a/\b) -- (\a/\prev);
                    \fi
                \fi
            }
        }
        
        \pgfsetfillopacity{1}
        \pgfsetstrokeopacity{1}
        \pgfsetlinewidth{0.5mm}
        \draw[red, dotted] (-8,-7) -- (-6,-5) -- (-5,-6) -- (-4,-5) -- (-2,-7);
    \end{tikzpicture}
    \caption{The Hasse diagram of $\mathcal{P}_{7,10}$ with the order ideal $\{11,8,4,1\}$ indicated}
    \label{fig-p}
\end{figure}
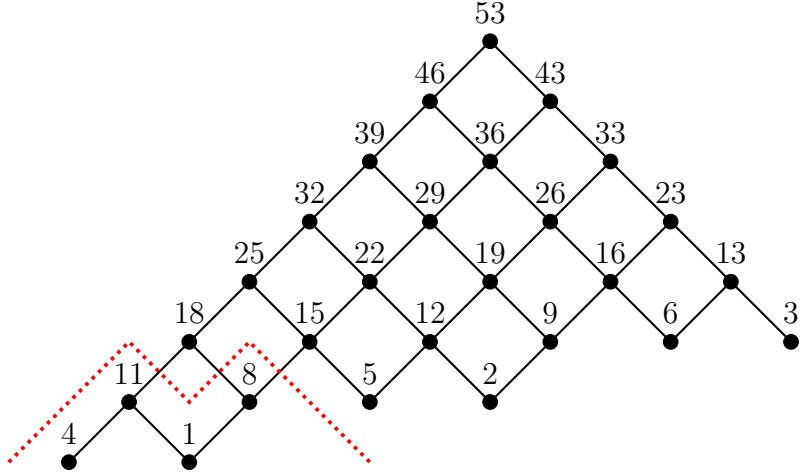

Recall that if $s$ and $k$ are coprime, then the greatest element of $\mathcal{P}_{s, s+k}$ is $M = s(s+k) - s - (s+k)$ \cite{sylvester1882}. Further, every $x \in \mathcal{P}_{s, s+k}$ can be uniquely written as
\[
    x = M - as - b(s+k),
\]
where $a,b \in \Z_{\ge 0}$ \cite[Lemma~3.1]{bny2018}.

\section{The coprime case} \label{sec-coprime}

In what follows, $s,k,d \in \Z_{>0}$ with $s$ and $k$ coprime and $s \ge 2$. We write $\mathcal{P}$ instead of $\mathcal{P}_{s, s+k}$.

The proof of Theorem~\ref{thm-max-hook-coprime} proceeds in two steps. First, in Section~\ref{subsec-interval-ideal}, we reduce the problem of finding the maximum possible hook length to that of finding the best strip along the bottom of $\mathcal{P}$ (what we will call an \emph{interval ideal}) according to two criteria. Then, in Section~\ref{subsec-best-interval-ideal}, we determine the best strip.

\subsection{Reduction to interval ideals} \label{subsec-interval-ideal}

We begin by defining the bottom of $\mathcal{P}$, which we call $\mathcal{E}$, and we impose an order on it.

\begin{definition}
    Let $\mathcal{E} = \mathcal{P} \cap [s+k-1]$. For $x,y\in \mathcal{E}$, let $x \lessdot_{\mathcal{E}} y$ if $y=x+s$ or $y=x-k$. Then, $<_{\mathcal{E}}$ is the transitive closure of $\lessdot_{\mathcal{E}}$.
\end{definition}

Figure~\ref{fig-e} illustrates $\mathcal{P}_{7,10}$ with $\mathcal{E}$ highlighted blue. The order on $\mathcal{E}$ is the left-to-right order in the figure. Thus, one expects that the order on $\mathcal{E}$ is total, which we now prove.

\begin{figure}[!ht]
    \centering
    \begin{tikzpicture}[scale=0.8]
        \def\s{7}
        \def\t{10}
        \def\F{\s*\t-\s-\t}
        \foreach \a in {0, 1,...,\t} {
            \foreach \b in {0, 1,...,\s} {
                \pgfmathparse{int(\F-\a*\s-\b*\t)}
                \ifnum \pgfmathresult>0
                    \node[label={\pgfmathresult}] [dot] (\a/\b) at (\b-\a,-\a-\b) {};
                    \ifnum \a > 0
                        \pgfmathtruncatemacro{\prev}{\a-1}
                        \draw [thick] (\a/\b) -- (\prev/\b);
                    \fi
                    \ifnum \b > 0
                        \pgfmathtruncatemacro{\prev}{\b-1}
                        \draw [thick] (\a/\b) -- (\a/\prev);
                    \fi
                \fi
            }
        }
        
        % drawing E
        \pgfsetstrokeopacity{0.3}
        \pgfsetlinewidth{6mm}
        
        \pgfmathparse{floor((\F-1)/\s)}
        \edef\l{\pgfmathresult}
        \pgfmathparse{floor((\F-1)/\t)}
        \edef\r{\pgfmathresult}
        
        \draw[blue] (-\l-0.5,-\l) -- (-5,-\l) -- (-4, -\l+1) -- (0, -\l+1) -- (1, -\r) -- (\r+0.5, -\r);
    \end{tikzpicture}
    \caption{The Hasse diagram of $\mathcal{P}_{7,10}$ with $\mathcal{E}$ highlighted blue}
    \label{fig-e}
\end{figure}
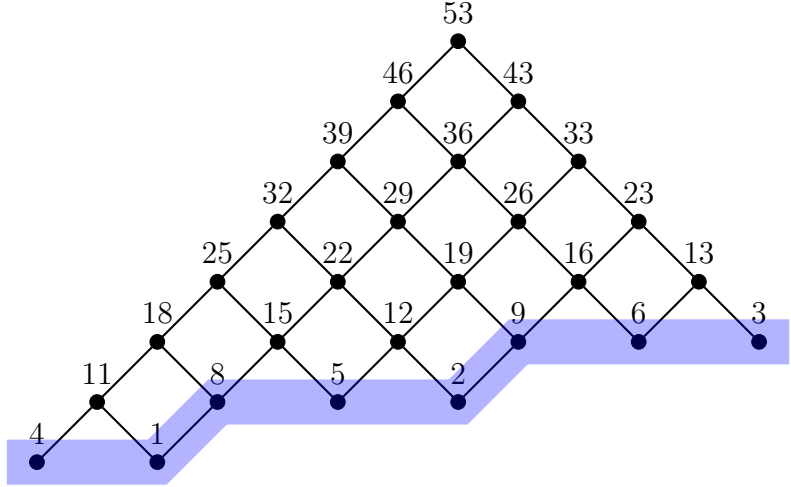

\begin{lemma}
    The order on $\mathcal{E}$ is total.
\end{lemma}
\begin{proof}
    We first prove that $x \not<_{\mathcal{E}} x$ for all $x \in \mathcal{E}$. Suppose for the sake of contradiction that $x <_{\mathcal{E}} x$, so for some sequence of $x_i \in \mathcal{E}$ and $n \ge 2$,
    \[
        x = x_1 \lessdot_{\mathcal{E}} x_2 \lessdot_{\mathcal{E}} \dots \lessdot_{\mathcal{E}} x_n = x.
    \]
    We may assume that $x_1, x_2, \dots, x_{n-1}$ are distinct. Then, $x + as - bk = x$, where
    \begin{align*}
        a &= |\{i \in [n-1] \mid x_{i+1} = x_i + s\}| \quad\text{and} \\
        b &= |\{i \in [n-1] \mid x_{i+1} = x_i - k\}|.
    \end{align*}
    Since $s$ and $k$ are coprime, $k \mid a$. But $x_{i+1} = x_i + s$ implies that $x_{i+1} \in [s+1,s+k-1]$. Thus, $a \le k-1$. It follows that $a=b=0$, a contradiction.
    
    Now, observe that for all $x \in \mathcal{E}$, we have $x+s \in \mathcal{E}$ if and only if $x < k$, and $x-k \in \mathcal{E}$ if and only if $x > k$. Thus, $k$ is the unique maximal element with respect to the order on $\mathcal{E}$. Next, observe that for all $x \in \mathcal{E}$, we have $x-s \in \mathcal{E}$ only if $x > s-1$, and $x+k \in \mathcal{E}$ only if $x \le s-1$. Thus, there is at most one $y \in \mathcal{E}$ with $y \lessdot_{\mathcal{E}} x$. These two facts imply the lemma.
\end{proof}

Next, we define two functions on elements of $\mathcal{P}$.

\begin{definition} \label{def-h}
    Given $x \in \mathcal{P}$, let
    \[
        h(x)=\floor{\frac{x}{s}} + 1.
    \]
\end{definition}

\begin{definition} \label{def-g}
    Given $x \in \mathcal{P}$, let $g(x) = x-(h(x)-1)s = x \bmod s$.
\end{definition}

Intuitively, $h(x)$ measures how long $\ang{x} \cap \mathcal{E}$ is. For example, if $s=7$ and $k=3$, then $h(19) = 3 = |\ang{19} \cap \mathcal{E}|$ as shown in Figure~\ref{fig-hg}. We think of $g(x)$ as the first element of $\ang{x} \cap \mathcal{E}$. If $s=7$ and $k=3$, then $g(19) = 5$, which is the first element of $\ang{19} \cap \mathcal{E}$ as shown in the figure. We now prove these interpretations of $h$ and $g$.

\begin{figure}[!ht]
    \centering
    \begin{tikzpicture}[scale=0.8]
        \def\s{7}
        \def\t{10}
        \def\F{\s*\t-\s-\t}
        \foreach \a in {0, 1,...,\t} {
            \foreach \b in {0, 1,...,\s} {
                \pgfmathparse{int(\F-\a*\s-\b*\t)}
                \ifnum \pgfmathresult>0
                    \node[label={\pgfmathresult}] [dot] (\a/\b) at (\b-\a,-\a-\b) {};
                    \ifnum \a > 0
                        \pgfmathtruncatemacro{\prev}{\a-1}
                        \draw [thick] (\a/\b) -- (\prev/\b);
                    \fi
                    \ifnum \b > 0
                        \pgfmathtruncatemacro{\prev}{\b-1}
                        \draw [thick] (\a/\b) -- (\a/\prev);
                    \fi
                \fi
            }
        }
        
        % drawing E
        \pgfsetstrokeopacity{0.3}
        \pgfsetlinewidth{6mm}
        
        \pgfmathparse{floor((\F-1)/\s)}
        \edef\l{\pgfmathresult}
        \pgfmathparse{floor((\F-1)/\t)}
        \edef\r{\pgfmathresult}
        
        \draw[blue] (-2.5, -\l+1) -- (0, -\l+1) -- (1, -\r) -- (1.5, -\r);
        
        % drawing <H>
        \pgfsetfillopacity{1}
        \pgfsetstrokeopacity{1}
        \pgfsetlinewidth{0.5mm}
        \draw[red, dotted] (-3,-6) -- (0,-3);
        \draw[red, dotted] (0,-3) -- (3,-6);
    \end{tikzpicture}
    \caption{The Hasse diagram of $\mathcal{P}_{7,10}$ with $\ang{19}$ indicated and $\ang{19} \cap \mathcal{E}$ highlighted blue}
    \label{fig-hg}
\end{figure}
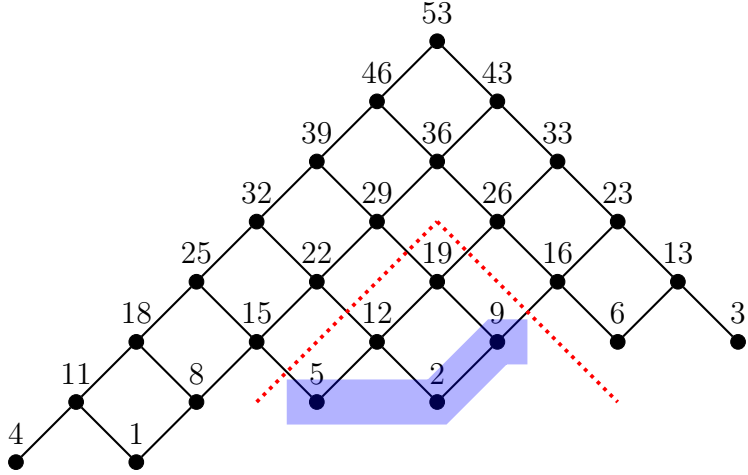

\begin{lemma} \label{lem-hx}
    For all $x \in \mathcal{P}$, we have $h(x) = |\ang{x} \cap \mathcal{E}|$.
\end{lemma}
\begin{proof}
    Consider the set $A = \{x - as \mid a \in [0,h(x)-1]\}$. It suffices to prove that the map
    \begin{align*}
        f: A &\to \ang{x} \cap \mathcal{E} \\
        y &\mapsto y - \floor{\frac{y}{s+k}}(s+k) = y \bmod (s+k)
    \end{align*}
    is a bijection.
    
    We first prove that $f$ is injective. Every $z \in \ang{x}$ can be uniquely written as
    \[
        z = x - as - b(s+k),
    \]
    where $a,b \in \Z_{\ge 0}$. The elements of $A$ have distinct $s$-coefficients, and $f(y)$ has the same $s$-coefficient as $y$. It follows that $f$ is injective.
    
    It remains to prove that $f$ is surjective. Let
    \[
        z = x - as - b(s+k) \in \ang{x} \cap \mathcal{E},
    \]
    where $a,b \in \Z_{\ge 0}$. Then, $f(x - as) = z$. It follows that $f$ is surjective.
\end{proof}

\begin{lemma} \label{lem-gx}
    For all $x \in \mathcal{P}$, we have $g(x)$ is the first element of $\ang{x} \cap \mathcal{E}$ with respect to the order on $\mathcal{E}$.
\end{lemma}
\begin{proof}
    Let $y$ be the first element of $\ang{x} \cap \mathcal{E}$. Then, $y$ can be uniquely written as
    \[
        y = x - as - b(s+k),
    \]
    where $a,b \in \Z_{\ge 0}$. We have $y \le s-1$; otherwise,
    \[
        x - (a+1)s-b(s+k) = y-s <_{\mathcal{E}} y,
    \]
    a contradiction. We also have $b = 0$; otherwise,
    \[
        x - (a+1)s - (b-1)(s+k) = y+k <_{\mathcal{E}} y,
    \]
    a contradiction. It follows that $a = \floor{x/s} = h(x)-1$, so $y = g(x)$, as desired.
\end{proof}

The importance of $h$ and $g$ lies in the following simple observation.

\begin{lemma} \label{lem-hg-lex}
    For all $x,y \in \mathcal{P}$, we have $x < y$ if and only if $(h(x), g(x)) \prec (h(y), g(y))$, where $\prec$ is the lexicographic order.
\end{lemma}
\begin{proof}
    This is clear from $x = g(x) + (h(x)-1)s$, viewing $h(x)-1$ and $g(x)$ as the quotient and remainder respectively of Euclidean division of $x$ by $s$.
\end{proof}

We now define a special kind of strip along $\mathcal{E}$.

\begin{definition}
    We say that $\mathcal{I} \subseteq \mathcal{E}$ is an \emph{interval ideal} if $\mathcal{I}$ is an interval with respect to the order on $\mathcal{E}$ and $\mathcal{I}$ is an order ideal of $\mathcal{P}$.
\end{definition}

The heart of this subsection is the following lemma, which gives the correspondence between elements of $\mathcal{P}$ and interval ideals.

\begin{lemma} \label{lem-interval-bijection}
    Let $\mathfrak{E}$ be the set of nonempty interval ideals. Then, the map
    \begin{align*}
        \pi: \mathcal{P} &\to \mathfrak{E} \\
        x &\mapsto \ang{x} \cap \mathcal{E}
    \end{align*}
    is a bijection. Further, $\ang{x}$ is $d$-distinct if and only if $\ang{x} \cap \mathcal{E}$ is $d$-distinct.
\end{lemma}
\begin{proof}
    We first prove that $\ang{x} \cap \mathcal{E}$ is a nonempty interval ideal. Since $\ang{x}$ is non-empty, it must have a minimal element with respect to the order on $\mathcal{P}$. Thus, $\ang{x} \cap \mathcal{E}$ is nonempty. Since $\ang{x}$ and $\mathcal{E}$ are order ideals of $\mathcal{P}$, we have that $\ang{x} \cap \mathcal{E}$ is an order ideal of $\mathcal{P}$. Recall from Lemma~\ref{lem-hx} that $f(A) = \ang{x} \cap \mathcal{E}$. Thus, to prove that $\ang{x} \cap \mathcal{E}$ is an interval with respect to the order on $\mathcal{E}$, it suffices to prove that $f(x-(a+1)s) \lessdot_{\mathcal{E}} f(x-as)$ for all $a \in [0,h(x)-2]$. If $f(x-as) \le s-1$, then
    \[
        f(x-(a+1)s) = f(x-as)-s+(s+k) = f(x-as)+k \lessdot_{\mathcal{E}} f(x-as). 
    \]
    If $f(x-as) > s-1$, then
    \[
        f(x-(a+1)s) = f(x-as)-s \lessdot_{\mathcal{E}} f(x-as).
    \]
    
    We now prove that $\pi$ is injective. Let $\mathcal{I}$ be a nonempty interval ideal. By Lemmas~\ref{lem-hx} and \ref{lem-gx}, $\mathcal{I}$ uniquely determines $h(x)$ and $g(x)$ for any $x$ with $\pi(x) = \mathcal{I}$. But then, $\mathcal{I}$ uniquely determines $x = g(x) + (h(x)-1)s$, so $\pi$ is injective. Since $x$ is the join of the first and last elements of $\mathcal{I}$, we have $x \in \mathcal{P}$. Then, $\pi(x) = \mathcal{I}$, so $\pi$ is surjective.
    
    It remains to prove that $\ang{x}$ is $d$-distinct if and only if $\ang{x} \cap \mathcal{E}$ is $d$-distinct. It is clear that if $\ang{x}$ is $d$-distinct, then $\ang{x} \cap \mathcal{E}$ is $d$-distinct. Conversely, suppose $\ang{x}$ is not $d$-distinct. Let $y,z \in \ang{x}$ with $0 < |y-z| \le d$. If $y, z > s-1$, then $y-s, z-s \in \ang{x}$ with $0 < |(y-s)-(z-s)| \le d$. Thus, we may assume that $y \le s-1$ or $z \le s-1$. Without loss of generality, assume that $y \le s-1$. If $z \le s+k-1$, then $y,z \in \ang{x} \cap \mathcal{E}$, so $\ang{x} \cap \mathcal{E}$ is not $d$-distinct. If $z > s+k-1$, then $d > k$, in which case any two adjacent elements of $\mathcal{E}$ that differ by $k$ are within $d$ of each other. Since $x > s+k-1$ in this case, $\ang{x} \cap \mathcal{E}$ is not $d$-distinct, as desired.
\end{proof}

The following lemma completes the reduction to interval ideals.

\begin{lemma} \label{lem-interval-lex}
    We have $\ang{H_d} \cap \mathcal{E}$ is the interval ideal $\mathcal{I}$ maximizing $(|\mathcal{I}|, \mathcal{I}_1)$ lexicographically over all $d$-distinct interval ideals, where $\mathcal{I}_1$ is the first element of $\mathcal{I}$ with respect to the order on $\mathcal{E}$.
\end{lemma}
\begin{proof}
    This is immediate from Lemmas~\ref{lem-hg-lex}, \ref{lem-hx}, \ref{lem-gx}, and \ref{lem-interval-bijection}.
\end{proof}

\subsection{Finding the best interval ideal} \label{subsec-best-interval-ideal}

By Lemma~\ref{lem-interval-lex}, our goal is now to find the longest interval ideal, using the magnitude of its first element as a tiebreaker.

First, we partition the elements of $\mathcal{E}$ according to their residue classes modulo $k$.

\begin{definition}
    The \emph{ledge} $\mathcal{L}_i$ is the set
    \[
         \mathcal{L}_i = \{x \in \mathcal{E} \mid x \equiv i\pmod k\}.
    \]
\end{definition}

Figure~\ref{fig-ledges} illustrates $\mathcal{P}_{7,10}$ with its ledges color-coded. We see that $\mathcal{L}_1$ is red, $\mathcal{L}_2$ is green, and $\mathcal{L}_0$ is blue. In general, $\mathcal{L}_i$ immediately precedes $\mathcal{L}_{i+\sbar}$, unless $i \equiv 0 \pmod k$, in which case $\mathcal{L}_i$ is the last ledge in $\mathcal{P}$.

The following lemma gives the size of each ledge.

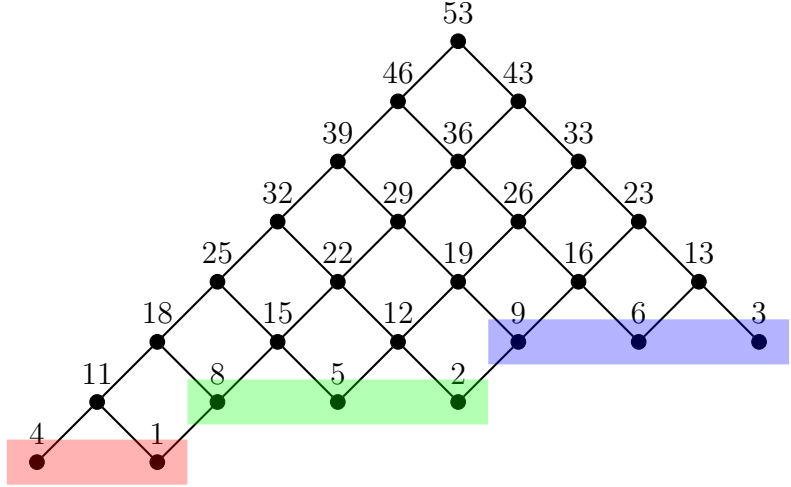
\begin{figure}[!ht]
    \centering
    \begin{tikzpicture}[scale=0.8]
        \def\s{7}
        \def\t{10}
        \def\F{\s*\t-\s-\t}
        \foreach \a in {0, 1,...,\t} {
            \foreach \b in {0, 1,...,\s} {
                \pgfmathparse{int(\F-\a*\s-\b*\t)}
                \ifnum \pgfmathresult>0
                    \node[label={\pgfmathresult}] [dot] (\a/\b) at (\b-\a,-\a-\b) {};
                    \ifnum \a > 0
                        \pgfmathtruncatemacro{\prev}{\a-1}
                        \draw [thick] (\a/\b) -- (\prev/\b);
                    \fi
                    \ifnum \b > 0
                        \pgfmathtruncatemacro{\prev}{\b-1}
                        \draw [thick] (\a/\b) -- (\a/\prev);
                    \fi
                \fi
            }
        }
        
        % drawing ledges
        \pgfsetstrokeopacity{0.3}
        \pgfsetlinewidth{6mm}
        
        \pgfmathparse{floor((\F-1)/\s)}
        \edef\l{\pgfmathresult}
        \pgfmathparse{floor((\F-1)/\t)}
        \edef\r{\pgfmathresult}
            
        \draw[red] (-\l-0.5,-\l) -- (-4.5,-\l);
        \draw[green] (-4.5, -\l+1) -- (0.5, -\l+1);
        \draw[blue] (0.5, -\r) -- (\r+0.5, -\r);
    \end{tikzpicture}
    \caption{The Hasse diagram of $\mathcal{P}_{7,10}$ with its ledges color-coded}
    \label{fig-ledges}
\end{figure}

\begin{lemma} \label{lem-ledge-lengths}
    For all $i \in [0,k-1]$, we have
    \[
        \abs{\mathcal{L}_i}=\begin{dcases}
            0 &\quad\text{if} \ s \mid i \ \text{and} \ i > 0\\
            \left\lfloor\frac{s-1}{k}\right\rfloor &\quad\text{if} \ i = \sbar\\
            1 &\quad\text{if} \ i=\ceil{k/s}s \bmod k \ \text{and} \ k>s\\
            \left\lfloor\frac{s-1}{k}\right\rfloor + 1 &\quad\text{if} \ i=0 \ \text{and} \ k > 1 \\
            \left\lfloor\frac{s-1}{k}\right\rfloor+1 &\quad\text{if} \ \sbar < i \ \text{and} \ s \nmid i\\
            \left\lfloor\frac{s-1}{k}\right\rfloor+2 &\quad\text{if} \ 0 < i < \sbar \ \text{and} \ i \neq \ceil{k/s}s \bmod k.
        \end{dcases}
    \]
\end{lemma}
\begin{proof}
    \textbf{Case I:} $s \mid i \ \text{and} \ i > 0$. In this case, $i \notin \mathcal{P}$. Then, $i+k$ is also a linear combination of $s$ and $s+k$, so $i+k \notin \mathcal{P}$. Since $s<k$, we have $i+2k \geq s+k$, so $i+2k\notin \mathcal{E}$. Then no integer congruent to $i \pmod k$ is in $\mathcal{E}$, and thus $|\mathcal{L}_i| = 0$.
    
    \textbf{Case II:} $i = \sbar$. First, suppose $k > 1$. We have
    \begin{align*}
         \sbar + \floor{\frac{s-1}{k}} k &= (\sbar-1) + \floor{\frac{s-1}{k}} k + 1 \\
         &= (s-1 \bmod k) + \floor{\frac{s-1}{k}} k + 1 \\
         &= (s-1) + 1 = s \notin \mathcal{E}.
    \end{align*}
    Thus, for all $b > \floor{(s-1)/k}$, we have $\sbar + bk \ge s+k$, so $\sbar + bk \notin \mathcal{E}$. And for $0 \le b < \floor{(s-1)/k}$, we have $0 < \sbar + bk < s$, so $\sbar + bk \in \mathcal{E}$. Thus, $|\mathcal{L}_i| = \floor{(s-1)/k}$.
    
    If $k=1$, then
    \[
        \sbar + \left(\floor{\frac{s-1}{k}} + 1\right) k = s \notin \mathcal{E}.
    \]
    Thus, for all $b > \floor{(s-1)/k} + 1$, we have we have $\sbar + bk \ge s + k$, so $\sbar + bk \notin \mathcal{E}$. And for all $0 < b < \floor{(s-1)/k} + 1$, we have $0 < \sbar + bk < s$, so $\sbar + bk \in \mathcal{E}$. Thus, $|\mathcal{L}_i| = \floor{(s-1)/k}$.
    
    \textbf{Case III:} $i = \ceil{k/s}s \bmod k \ \text{and} \ k>s$. We have
    \[
        k < \ceil{\frac{k}{s}}s < k + s.
    \]
    Hence,
    \[
        0 < \ceil{\frac{k}{s}}s \bmod k < s,
    \]
    so $i \in \mathcal{E}$. Then, $i+k=\ceil{k/s}s$, so $i+k \notin \mathcal{P}$. We also have $i+2k > 2k > s+k$, so $i+2k \notin \mathcal{E}$. Thus, $|\mathcal{L}_i| = 1$.
    
    \textbf{Case IV:} $i = 0 \ \text{and} \ k > 1$. We have
    \[
        s-1 < \left(\floor{\frac{s-1}{k}} + 1\right)k \le s+k-1.
    \]
    In fact, $s \nmid \left(\floor{(s-1)/k} + 1\right)k$, because $s$ and $k$ are coprime and $\floor{(s-1)/k} + 1 < s$. Thus, for all $0 < b \le \floor{(s-1)/k} + 1$, we have $bk \in \mathcal{E}$. Further, if $b > \floor{(s-1)/k} + 1$, then $bk > s+k-1$, so $bk \notin \mathcal{E}$. Thus, $|\mathcal{L}_i| = \floor{(s-1)/k} + 1$.
    
    \textbf{Case V:} $\sbar < i \ \text{and} \ s \nmid i$. We have
    \[
        s < i + s - \sbar = i + \left(\frac{s-1}{k} - \frac{\sbar - 1}{k}\right)k = i + \floor{\frac{s-1}{k}}k < s+k-1.
    \]
    If $s < k$, then $s = \sbar$, so $s \nmid i + s - \sbar$. If $s > k$, then no integers strictly between $s$ and $s+k-1$ are multiples of $s$, so again $s \nmid i + s - \sbar$. In either case, $s \nmid i + \floor{(s-1)/k}k$. Thus, for all $0 \le b \le \floor{(s-1)/k}$, we have $i+bk \in \mathcal{E}$. Further, if $b > \floor{(s-1)/k}$, we have $i + bk > s+k-1$, so $i+bk \notin \mathcal{E}$. Thus, $|\mathcal{L}_i| = \floor{(s-1)/k} + 1$.
    
    \textbf{Case VI:} $0 < i < \sbar \ \text{and} \ i \neq \ceil{k/s}s \bmod k$. We have
    \[
        s < i + \left(\floor{\frac{s-1}{k}} + 1\right)k = i + \left(\frac{s-1}{k} - \frac{\sbar - 1}{k} + 1\right)k = i + s - \sbar + k < s + k.
    \]
    If $s \mid i + s - \sbar + k$, then $k > s$, so $s=\sbar$ and $s \mid i + k$. Hence, $i = \ceil{k/s}s - k = \ceil{k/s}s \bmod k$, a contradiction. Thus, $s \nmid i + \left(\floor{(s-1)/k} + 1\right)k$, so for all $0 \le b \le \floor{(s-1)/k} + 1$, we have $i + bk \in \mathcal{E}$. Further, if $b > \floor{(s-1)/k} + 1$, we have $i + bk > s + k - 1$, so $i + bk \notin \mathcal{E}$. Thus, $|\mathcal{L}_i| = \floor{(s-1)/k} + 2$.
\end{proof}

One should think of the first three cases of Lemma~\ref{lem-ledge-lengths} as edge cases. In these cases, the ledge is either empty or the first or last ledge in $\mathcal{P}$. The final two cases are the main cases. The upshot is that, ignoring edge cases, there are two kinds of ledges: short ledges and long ledges. Still ignoring edge cases, $\mathcal{L}_i$ is long according to whether $i \in [\sbar - 1]$.

Before proceeding, the following notation for an interval that wraps around modulo $k$ will be useful.

\begin{definition}
    Given $a,b \in \Z$, let
    \[
        (a,b)_k = \begin{cases}
            (a \bmod k,b \bmod k) &\quad\text{if} \ a \bmod k \le b \bmod k \\
            (a \bmod k,k-1] \cup [0,b \bmod k) &\quad\text{if} \ a \bmod k > b \bmod k,
        \end{cases}
    \]
    and similarly for closed and half-open intervals.
\end{definition}

We say that $\mathcal{L}_p$ and $\mathcal{L}_q$ are \emph{within $d$ of each other} if $p - q \in [-d,d]_k$. A first approximation of our strategy for finding the best interval ideal is to choose as many adjacent ledges as possible such that no two are within $d$ of each other. Later we will see that this isn't exactly right, but this approximation motivates the strategy.

The maximum number of adjacent ledges such that no two are within $d$ of each other is given by $\esye$ (pronounced \textsf{ES-yay}). A sequence of $\esye$ adjacent ledges has the form
\[
    \mathcal{L}_i, \mathcal{L}_{i+\sbar}, \dots, \mathcal{L}_{i + \sbar (\esye - 1)},
\]
which motivates the following definition.

\begin{definition}
    An \emph{$\esye$-interval} is a tuple of elements of $\Z/k\Z$ of the form 
    \[
        (i, i+\sbar, \dots, i + \sbar (\esye - 1))
    \]
    for some $i \in \Z/k\Z$.
\end{definition}

To find the best interval ideal, we need to know how many long ledges are in a given sequence of $\esye$ adjacent ledges. Using Lemma~\ref{lem-ledge-lengths}, and ignoring edge cases, this is the same as $|I \cap [\sbar-1]|$, where $I$ is the $\esye$-interval of ledge indices. The next lemma determines the size of this intersection.

\begin{lemma} \label{lem-esye-interval}
    Suppose $d < k$. Let $I_i = (i, i+\sbar, \dots, i + \sbar (\esye - 1))$ be an $\esye$-interval not containing both 0 and $\sbar$. Then,
    \[
        |I_i \cap [\sbar-1]| = \begin{dcases}
            \ceil{\frac{\sbar \esye}{k}} &\quad\text{if} \ i \in (\sbar - \sbar \esye, \sbar)_k \\
            \ceil{\frac{\sbar \esye}{k}}-1 &\quad\text{if} \ i \in [\sbar, \sbar - \sbar \esye]_k.
        \end{dcases}
    \]
    In particular,
    \[
        \max_I |I \cap [\sbar-1]| = \ceil{\frac{\sbar \esye - 1}{k}},
    \]
    where the maximum is taken over all $\esye$-intervals not containing both 0 and $\sbar$.
\end{lemma}
\begin{proof}
    We actually prove that for all $\esye$-intervals $I_i$,
    \[
        |I_i \cap [0,\sbar-1]| = \begin{dcases}
            \ceil{\frac{\sbar \esye}{k}} &\quad\text{if} \ i \in [\sbar - \sbar \esye, \sbar)_k \\
            \ceil{\frac{\sbar \esye}{k}}-1 &\quad\text{if} \ i \in [\sbar, \sbar - \sbar \esye)_k.
        \end{dcases}
    \]
    This implies the lemma, because if $I_i$ does not contain both 0 and $\sbar$, then $0 \in I_i$ if and only if $i = \sbar - \sbar \esye \bmod k$.
    
    Since $I_i = I_0 + i$, 
    \[
        |I_i \cap [0,\sbar-1]| - |I_0 \cap [0,\sbar-1]| = |I_0 \cap [-i, -1]_k| - |I_0 \cap [\sbar - i, \sbar - 1]_k|.
    \]
    If $x \in [-i, -1]_k$, then $x + \sbar \in [\sbar - i, \sbar - 1]_k$, so
    \[
        |I_0 \cap [-i, -1]_k| - |I_0 \cap [\sbar - i, \sbar - 1]_k| = \chi_{[-i,-1]_k}(\sbar(\esye - 1)) - \chi_{[\sbar - i,\sbar - 1]_k}(0) \eqqcolon \chi.
    \]
    We have $\sbar(\esye - 1) \in [-i,-1]_k$ if and only if $i \in [\sbar - \sbar \esye, 0]_k$, and $0 \in [\sbar - i,\sbar - 1]_k$ if and only if $i \in [\sbar, 0]_k$. Thus,
    \[
        \chi = \begin{cases}
            1 &\quad\text{if} \ i \in [\sbar - \sbar \esye, \sbar)_k \ \text{and} \ 0 < \sbar - \sbar \esye \bmod k < \sbar \bmod k \\
            0 &\begin{aligned}
                \quad\text{if}& \ (i \in [0,\sbar)_k \ \text{and} \ \sbar - \sbar \esye \bmod k = 0) \\
                \quad &\text{or} \ (i \in [\sbar, \sbar - \sbar \esye)_k \ \text{and} \ 0 < \sbar - \sbar \esye \bmod k < \sbar \bmod k) \\
                \quad &\text{or} \ (i \in [\sbar - \sbar \esye, \sbar)_k \ \text{and} \ \sbar \bmod k < \sbar - \sbar \esye \bmod k)
            \end{aligned} \\
            -1 &\begin{aligned}
                \quad\text{if}& \ (i \in [\sbar,0)_k \ \text{and} \ \sbar - \sbar \esye \bmod k = 0) \\
                \quad &\text{or} \ (i \in [\sbar, \sbar - \sbar \esye)_k \ \text{and} \ \sbar \bmod k < \sbar - \sbar \esye \bmod k).
            \end{aligned}
        \end{cases}
    \]
    In particular, $|I_i \cap [0,\sbar-1]| - |I_j \cap [0,\sbar-1]| \in \{-1, 0, 1\}$ for all $i$ and $j$. Since the average of $|I_i \cap [0,\sbar-1]|$ over all $i \in [0,k-1]$ is $\sbar \esye / k$, the lemma follows.
\end{proof}

We are finally ready to determine the best interval ideal.

\begin{lemma} \label{lem-best-interval-ideal}
    Suppose $d < k$. Then, $\ang{H_d} \cap \mathcal{E}$ is the interval ideal $\mathcal{I}$, where $\mathcal{I}$ contains the union of $\esye$ adjacent ledges beginning at $\mathcal{L}_i$---excluding the non-minimal element in $\mathcal{L}_i$, if any---and
    \[
        i = \begin{cases}
            \sbar-2 &\quad\text{if} \ \sbar \esye \bmod k = 1 \\
            \sbar-1 &\quad\text{if} \ 1 < \sbar \esye \bmod k \le d \ \text{or} \ d < \sbar \esye \bmod k = k - 1 \\
            \sbar-\sbar\esye-1 &\quad\text{if} \ d < \sbar \esye \bmod k < k-1.
        \end{cases}
    \]
    If $d < \sbar \esye \bmod k < k-1$, then $\mathcal{I}$ additionally contains the last element of $\mathcal{L}_{i-\sbar}$ and the first element of $\mathcal{L}_i$ with respect to the order on $\mathcal{E}$. If $\sbar \esye \bmod k = 1$ or $d < \sbar \esye \bmod k$, then $\mathcal{I}$ additionally contains the first element of $\mathcal{L}_{i+\sbar \esye}$ with respect to the order on $\mathcal{E}$. These are all the elements in $\mathcal{I}$.
\end{lemma}

Before proving the lemma, we give two examples. First, if $s = 7$, $k=3$, and $d=1$, then $\sbar\esye \bmod k = 1$. The lemma tells us that $\ang{H_1} \cap \mathcal{E}$ starts at the first non-minimal element of $\mathcal{L}_{\sbar-2} = \mathcal{L}_2$ and ends at the first element of $\mathcal{L}_0$. This example is illustrated in Figure~\ref{fig-hg}. Second, if $s = 8$, $k=5$, and $d=2$, then $\sbar\esye \bmod k = 3$. The lemma tells us that $\ang{H_2} \cap \mathcal{E}$ starts at the last element of $\mathcal{L}_{-\sbar\esye-1} = \mathcal{L}_1$ and ends at the first element of $\mathcal{L}_{\sbar-1} = \mathcal{L}_2$. This example is illustrated in Figure~\ref{fig-8-13}.

\begin{figure}[!ht]
    \centering
    \begin{tikzpicture}[scale=0.8]
        \def\s{8}
        \def\t{13}
        \def\F{\s*\t-\s-\t}
        \foreach \a in {0, 1,...,\t} {
            \foreach \b in {0, 1,...,\s} {
                \pgfmathparse{int(\F-\a*\s-\b*\t)}
                \ifnum \pgfmathresult>0
                    \node[label={\pgfmathresult}] [dot] (\a/\b) at (\b-\a,-\a-\b) {};
                    \ifnum \a > 0
                        \pgfmathtruncatemacro{\prev}{\a-1}
                        \draw [thick] (\a/\b) -- (\prev/\b);
                    \fi
                    \ifnum \b > 0
                        \pgfmathtruncatemacro{\prev}{\b-1}
                        \draw [thick] (\a/\b) -- (\a/\prev);
                    \fi
                \fi
            }
        }
        
        % drawing E
        \pgfsetstrokeopacity{0.3}
        \pgfsetlinewidth{6mm}
        
        \pgfmathparse{floor((\F-1)/\s)}
        \edef\l{\pgfmathresult}
        \pgfmathparse{floor((\F-1)/\t)}
        \edef\r{\pgfmathresult}
        
        \draw[blue] (-5.5, -\l+1) -- (-5, -\l+1) -- (-4, -\l+2) -- (-2, -\l+2) -- (-1, -\l+3) -- (-0.5, -\l+3);
        
        % drawing <H>
        \pgfsetfillopacity{1}
        \pgfsetstrokeopacity{1}
        \pgfsetlinewidth{0.5mm}
        \draw[red, dotted] (-6,-9) -- (-2,-5);
        \draw[red, dotted] (-2,-5) -- (2,-9);
    \end{tikzpicture}
    \caption{The Hasse diagram of $\mathcal{P}_{8,13}$ with $\ang{25}$ indicated and $\ang{25} \cap \mathcal{E}$ highlighted blue}
    \label{fig-8-13}
\end{figure}
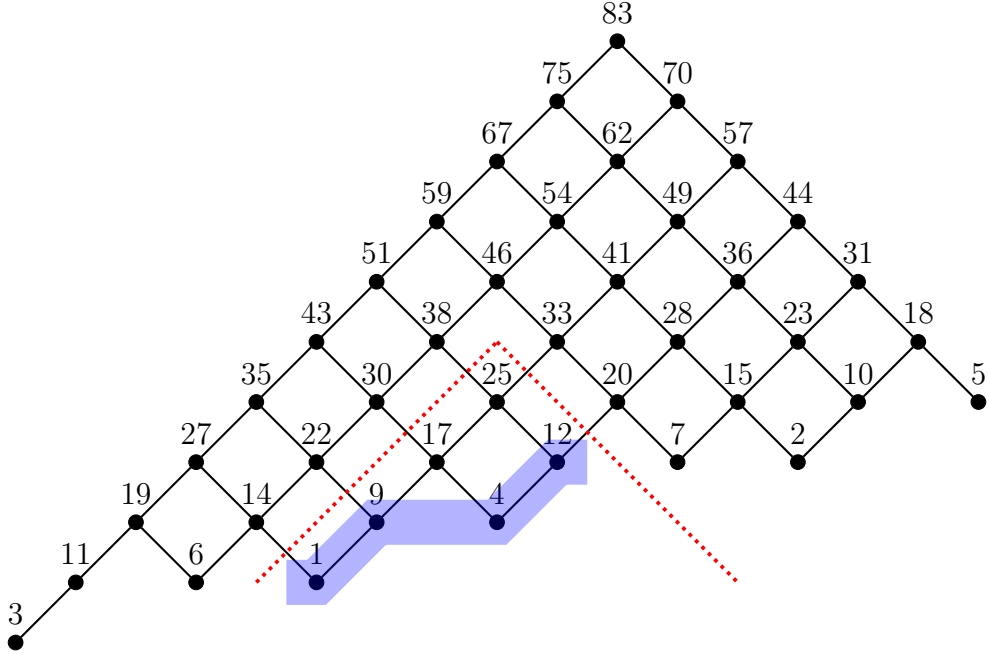

\begin{proof}[Proof of Lemma~\ref{lem-best-interval-ideal}]
    Throughout, we use $\mathcal{L}'_j$ to denote $\mathcal{L}_j$ excluding its non-minimal element, if any.
    
    \textbf{Case I:} $\sbar \esye \bmod k = 1$. In this case,
    \[
        \mathcal{I} = \mathcal{L}'_{\sbar-2} \cup \mathcal{L}_{2\sbar-2} \cup \dots \cup \mathcal{L}_{k-1} \cup \{y\},
    \]
    where $y$ is the first element of $\mathcal{L}_{\sbar-1}$. We first prove that $\mathcal{I}$ is a $d$-distinct interval ideal. Observe that $\mathcal{L}_{\sbar-2}$ and $\mathcal{L}_{\sbar-1}$ are the only ledges intersecting $\mathcal{I}$ that are within $d$ of each other. But $y = s+k-1$, and the greatest element of $\mathcal{L}'_{\sbar-2}$ is $s-2 < s+k-1-d$, so $\mathcal{I}$ is $d$-distinct. The observation also implies that $\mathcal{I} \cap \mathcal{L}_0 = \emptyset$; \emph{a fortiori}, $\mathcal{I}$ does not intersect both $\mathcal{L}_0$ and $\mathcal{L}_{\sbar}$. Thus, $\mathcal{I}$ is an interval with respect to the order on $\mathcal{E}$. Finally, $|\mathcal{L}'_{\sbar-2}| \ge 1$ by Lemma~\ref{lem-ledge-lengths}, so $\mathcal{I}$ is an order ideal of $\mathcal{P}$.
    
    We now prove that $\mathcal{I}$ maximizes $|\mathcal{I}|$ over all $d$-distinct interval ideals. Suppose for the sake of contradiction that there is a $d$-distinct interval ideal $\mathcal{I}'$ with $|\mathcal{I}'| > |\mathcal{I}|$. Let the first element of $\mathcal{I}'$ be the $r$th element of $\mathcal{L}'_j$. Then, by Lemmas~\ref{lem-ledge-lengths} and \ref{lem-esye-interval}, $\mathcal{I}'$ contains the $r$th element of $\mathcal{L}'_{j+1}$. Now, the $r$th element of $\mathcal{L}'_j$ is
    \begin{equation} \label{eqn-rth-elt}
        j + \floor{\frac{s-1-j}{k}}k - (r-1)k,
    \end{equation}
    and the $r$th element of $\mathcal{L}'_{j+1}$ is
    \[
        j+1 + \floor{\frac{s-2-j}{k}}k - (r-1)k.
    \]
    These differ by 1 unless $j+1 \equiv \sbar \pmod k$. If $j+1 \equiv \sbar \pmod k$, then $\mathcal{I}'$ intersects both $\mathcal{L}_0$ and $\mathcal{L}_{\sbar}$ and hence is not an interval with respect to the order on $\mathcal{E}$. Otherwise, $\mathcal{I}'$ is not $d$-distinct, a contradiction.
    
    By Lemma~\ref{lem-interval-lex}, it remains to prove that $\mathcal{I}$ maximizes $\mathcal{I}_1$ over all $d$-distinct interval ideals of size $|\mathcal{I}|$. We have $\mathcal{I}_1 = s - 2$. The only potentially greater value of $\mathcal{I}_1$ is $s-1$, but by Lemmas~\ref{lem-ledge-lengths} and \ref{lem-esye-interval}, an interval ideal $\mathcal{I}'$ with $\mathcal{I}'_1 = s-1$ must satisfy $|\mathcal{I}'| < |\mathcal{I}|$.
    
    \textbf{Case II:} $1 < \sbar \esye \bmod k \le d$. In this case,
    \[
        \mathcal{I} = \mathcal{L}'_{\sbar - 1} \cup \mathcal{L}_{2\sbar - 1} \cup \dots \cup \mathcal{L}_{\sbar \esye - 1}.
    \]
    We first prove that $\mathcal{I}$ is a $d$-distinct interval ideal. It does not intersect any ledges that are within $d$ of each other, so $\mathcal{I}$ is $d$-distinct. Hence, since $\mathcal{I} \cap \mathcal{L}_{\sbar-1} \neq \emptyset$, we have $\mathcal{I} \cap \mathcal{L}_{\sbar} = \emptyset$; \emph{a fortiori}, $\mathcal{I}$ does not intersect both $\mathcal{L}_0$ and $\mathcal{L}_{\sbar}$. Thus, $\mathcal{I}$ is an interval with respect to the order on $\mathcal{E}$. Finally, $|\mathcal{L}'_{\sbar-1}| \ge 1$ by Lemma~\ref{lem-ledge-lengths}, so $\mathcal{I}$ is an order ideal of $\mathcal{P}$.
    
    We now prove that $\mathcal{I}$ maximizes $|\mathcal{I}|$ over all $d$-distinct interval ideals. Suppose for the sake of contradiction that there is a $d$-distinct interval ideal $\mathcal{I}'$ with $|\mathcal{I}'| > |\mathcal{I}|$. Let the first element of $\mathcal{I}'$ be the $r$th element of $\mathcal{L}'_j$. If $j \in (\sbar - \sbar \esye, \sbar)_k$, then by Lemmas~\ref{lem-ledge-lengths} and \ref{lem-esye-interval}, $\mathcal{I}'$ contains the $r$th element of $\mathcal{L}_{j+\sbar \esye}$. Now, the $r$th element of $\mathcal{L}_{j+\sbar \esye}$ is
    \[
        j + \sbar \esye + \floor{\frac{s+k-1-j-\sbar\esye}{k}}k - (r-1)k.
    \]
    This differs from \eqref{eqn-rth-elt} by $\sbar \esye \bmod k$, so $\mathcal{I}'$ is not $d$-distinct, a contradiction in this case. If $j \in [\sbar, \sbar-\sbar\esye]_k$, then by Lemmas~\ref{lem-ledge-lengths} and \ref{lem-esye-interval}, $\mathcal{I}'$ contains the $r$th element of $\mathcal{L}'_{j+\sbar \esye}$. Now, the $r$th element of $\mathcal{L}'_{j+\sbar \esye}$ is
    \begin{equation} \label{eqn-rth-elt-2}
        j + \sbar \esye + \floor{\frac{s-1-j-\sbar\esye}{k}}k - (r-1)k.
    \end{equation}
    This differs from \eqref{eqn-rth-elt} by $\sbar \esye \bmod k$ unless $j+\sbar\esye \equiv \sbar \pmod k$. If $j+\sbar\esye \equiv \sbar \pmod k$, then $\mathcal{I}'$ intersects both $\mathcal{L}_0$ and $\mathcal{L}_{\sbar}$ and hence is not an interval with respect to the order on $\mathcal{E}$. Otherwise, $\mathcal{I}'$ is not $d$-distinct, a contradiction.
    
    By Lemma~\ref{lem-interval-lex}, it remains to prove that $\mathcal{I}$ maximizes $\mathcal{I}_1$ over all $d$-distinct interval ideals of size $|\mathcal{I}|$. This follows from the fact that $\mathcal{I}_1$ is the first element of $\mathcal{L}'_{\sbar-1}$, which is $s-1$.
    
    \textbf{Case III:} $d < \sbar \esye \bmod k = k - 1$. In this case,
    \[
        \mathcal{I} = \mathcal{L}'_{\sbar - 1} \cup \mathcal{L}_{2\sbar - 1} \cup \dots \cup \mathcal{L}_{k-2} \cup \{y\},
    \]
    where $y$ is the first element of $\mathcal{L}_{\sbar-2}$. We first prove that $\mathcal{I}$ is a $d$-distinct interval ideal. Observe that $\mathcal{L}_{\sbar-1}$ and $\mathcal{L}_{\sbar-2}$ are the only ledges intersecting $\mathcal{I}$ that are within $d$ of each other. But $y = s+k-2$, and the greatest element of $\mathcal{L}'_{\sbar-1}$ is $s-1 < s+k-2-d$, so $\mathcal{I}$ is $d$-distinct. The observation also implies that $\mathcal{I} \cap \mathcal{L}_{\sbar} = \emptyset$; \emph{a fortiori}, $\mathcal{I}$ does not intersect both $\mathcal{L}_0$ and $\mathcal{L}_{\sbar}$. Thus, $\mathcal{I}$ is an interval with respect to the order on $\mathcal{E}$. Finally, $|\mathcal{L}'_{\sbar-1}| \ge 1$ by Lemma~\ref{lem-ledge-lengths}, so $\mathcal{I}$ is an order ideal of $\mathcal{P}$.
    
    We now prove that $\mathcal{I}$ maximizes $|\mathcal{I}|$ over all $d$-distinct interval ideals. Suppose for the sake of contradiction that there is a $d$-distinct interval ideal $\mathcal{I}'$ with $|\mathcal{I}'| > |\mathcal{I}|$. Let the first element of $\mathcal{I}'$ be the $r$th element of $\mathcal{L}'_j$. If $j \in (\sbar + 1, \sbar)_k$, then by Lemmas~\ref{lem-ledge-lengths} and \ref{lem-esye-interval}, $\mathcal{I}'$ contains the $r$th element of $\mathcal{L}'_{j-1}$. Now, the $r$th element of $\mathcal{L}'_{j-1}$ is
    \[
        j-1 + \floor{\frac{s-j}{k}}k - (r-1)k.
    \]
    This differs from \eqref{eqn-rth-elt} by 1, so $\mathcal{I}'$ is not $d$-distinct, a contradiction in this case. If $j \in [\sbar, \sbar+1]_k$, then by Lemmas~\ref{lem-ledge-lengths} and \ref{lem-esye-interval}, $\mathcal{I}'$ contains the $(r+1)$th element of $\mathcal{L}'_{j-1}$, which differs from \eqref{eqn-rth-elt} by 1 unless $j-1 \equiv \sbar \pmod k$. If $j-1 \equiv \sbar \pmod k$, then $\mathcal{I}'$ intersects both $\mathcal{L}_0$ and $\mathcal{L}_{\sbar}$ and hence is not an interval with respect to the order on $\mathcal{E}$. Otherwise, $\mathcal{I}'$ is not $d$-distinct, a contradiction.
    
    By Lemma~\ref{lem-interval-lex}, it remains to prove that $\mathcal{I}$ maximizes $\mathcal{I}_1$ over all $d$-distinct interval ideals of size $|\mathcal{I}|$. This follows from the fact that $\mathcal{I}_1$ is the first element of $\mathcal{L}'_{\sbar-1}$, which is $s-1$.
    
    \textbf{Case IV:} $d < \sbar \esye \bmod k < k-1$. In this case,
    \[
        \mathcal{I} = \{x\} \cup \mathcal{L}_{\sbar - \sbar \esye - 1} \cup \mathcal{L}_{2\sbar - \sbar \esye - 1} \cup \dots \cup \mathcal{L}_{k-1} \cup \{y\},
    \]
    where $x$ is the last element of $\mathcal{L}_{-\sbar\esye-1}$ and $y$ is the first element of $\mathcal{L}_{\sbar -1}$. We first prove that $\mathcal{I}$ is a $d$-distinct interval ideal. Observe that $\{\mathcal{L}_{-\sbar\esye-1}, \mathcal{L}_{k-1}\}$, $\{\mathcal{L}_{\sbar - \sbar \esye - 1}, \mathcal{L}_{\sbar -1}\}$, and possibly $\{\mathcal{L}_{-\sbar\esye-1}, \mathcal{L}_{\sbar -1}\}$ are the only pairs of ledges intersecting $\mathcal{I}$ that are within $d$ of each other. But $x = -\sbar\esye - 1 \bmod k = k-1 - (\sbar\esye \bmod k)$, and the least element of $\mathcal{L}_{k-1}$ is $k-1 > k-1 - (\sbar\esye \bmod k) + d$. Similarly, $y = s+k-1$, and the greatest element of $\mathcal{L}_{\sbar - \sbar \esye - 1}$ is $s+k-1-(\sbar\esye \bmod k) < s+k-1-d$. Finally, $k-1 - (\sbar\esye \bmod k) + d < s+k-1$, so $\mathcal{I}$ is $d$-distinct. The observation also implies that $\mathcal{I} \cap \mathcal{L}_0 = \emptyset$; \emph{a fortiori}, $\mathcal{I}$ does not intersect both $\mathcal{L}_0$ and $\mathcal{L}_{\sbar}$. Thus, $\mathcal{I}$ is an interval with respect to the order on $\mathcal{E}$. Finally, $|\mathcal{L}'_{-\sbar\esye-1}| \ge 1$ by Lemma~\ref{lem-ledge-lengths}, so $\mathcal{I}$ is an order ideal of $\mathcal{P}$.
    
    Consider a $d$-distinct interval ideal $\mathcal{I}'$ with $|\mathcal{I}'| \ge |\mathcal{I}|$. Let the first element of $\mathcal{I}'$ be the $r$th element of $\mathcal{L}'_j$. We claim that $j \in (0, -\sbar\esye)_k$ and $|\mathcal{L}'_j| = r$. Suppose not. If $j \in (\sbar - \sbar \esye, \sbar)_k$, then by Lemmas~\ref{lem-ledge-lengths} and \ref{lem-esye-interval}, $\mathcal{I}'$ contains the $r$th element of $\mathcal{L}'_{j+\sbar\esye}$. Now, \eqref{eqn-rth-elt} and \eqref{eqn-rth-elt-2} differ by $k - (\sbar\esye \bmod k) \le d$, so $\mathcal{I}'$ is not $d$-distinct, a contradiction in this case. If $j \in [\sbar, \sbar-\sbar\esye)_k$, then by Lemmas~\ref{lem-ledge-lengths} and \ref{lem-esye-interval}, $\mathcal{I}'$ contains the $(r+1)$th element of $\mathcal{L}'_{j+\sbar\esye}$. Now, the $(r+1)$th element of $\mathcal{L}'_{j+\sbar\esye}$ is
    \[
        j + \sbar \esye + \floor{\frac{s-1-j-\sbar\esye}{k}}k - rk.
    \]
    This differs from \eqref{eqn-rth-elt} by $k - (\sbar\esye \bmod k) \le d$, so $\mathcal{I}'$ is not $d$-distinct, a contradiction in this case. Finally, if $j \equiv \sbar - \sbar\esye \pmod k$, then $\mathcal{I}'$ intersects both $\mathcal{L}_0$ and $\mathcal{L}_{\sbar}$ and hence is not an interval with respect to the order on $\mathcal{E}$, a contradiction.
    
    We now prove that $\mathcal{I}$ maximizes $|\mathcal{I}|$ over all $d$-distinct interval ideals. Suppose for the sake of contradiction that there is a $d$-distinct interval ideal $\mathcal{I}'$ with $|\mathcal{I}'| > |\mathcal{I}|$. By the claim, the first element of $\mathcal{I}'$ is the last element of $\mathcal{L}_j$, where $j \in (0, -\sbar\esye)_k$. Then, by Lemmas~\ref{lem-ledge-lengths} and \ref{lem-esye-interval}, $\mathcal{I}'$ contains the first element of $\mathcal{L}'_{j + \sbar + \sbar\esye}$. Now, the first element of $\mathcal{L}'_{j + \sbar + \sbar\esye}$ is
    \[
        j + \sbar + \sbar\esye + \floor{\frac{s-1-j-\sbar-\sbar\esye}{k}}k.
    \]
    We also have that $\mathcal{I}'$ contains the first element of $\mathcal{L}_{j + \sbar}$, which is
    \[
        j + \sbar + \floor{\frac{s+k-1-j-\sbar}{k}}k.
    \]
    These differ by $k - (\sbar\esye \bmod k) \le d$, so $\mathcal{I}'$ is not $d$-distinct, a contradiction.
    
    By Lemma~\ref{lem-interval-lex}, it remains to prove that $\mathcal{I}$ maximizes $\mathcal{I}_1$ over all $d$-distinct interval ideals of size $|\mathcal{I}|$. We have $\mathcal{I}_1 = x = k-1-(\sbar\esye \bmod k)$, which is maximal by the claim.
\end{proof}

We now prove Theorem~\ref{thm-max-hook-coprime}, which says that the maximum possible hook length $H_d$ of an $(s,s+k)$-core with $d$-distinct parts is
\[
    H_d(s,k) = \begin{cases}
        s-1 &\quad\text{if} \ k = 1 \ \text{or} \ k,s \leq d \\
        s+k-1 &\quad\text{if} \ 1 < k \le d < s \\
        B-2 &\quad\text{if} \ d < k \ \text{and} \ \sbar \esye \bmod k = 1 \\
        B-s-1 &\quad\text{if} \ 1 < \sbar \esye \bmod k \le d < k\\
        B + k - \sbar \esye -1 &\quad\text{if} \ d < \sbar \esye \bmod k < k-1 \\
        B-1 &\quad\text{if} \ d < \sbar \esye \bmod k = k - 1,
    \end{cases}
\]
where
\begin{align*}
    B &= \floor{\frac{s-1}{k}} \left(k + s\esye\right) + s\left( \ceil{\frac{\sbar\esye-1}{k}} + \esye - 1\right)+\sbar, \\
    \sbar &= s \bmod k, \ \text{and} \\
    \esye &= \min \{\ell \cdot (\sbar)^{-1} \bmod k \mid -d \leq \ell \leq d, \ \ell\neq0\}.
\end{align*}
\begin{proof}[Proof of Theorem~\ref{thm-max-hook-coprime}]
    \textbf{Case I:} $k = 1 \ \text{or} \ k,s \leq d$. If $k=1$, adjacent elements of $\mathcal{E}$ are within $d$ of each other, so $\ang{H_d}$ can only have one element in $\mathcal{E}$. Since $s-1$ is the greatest element with this property (given that it is the greatest element of $\mathcal{E}$), $H_d = s-1$.
    
    If $k,s \le d$, adjacent elements of $\mathcal{E}$ are within $d$ of each other, and any element of $\mathcal{P}$ is within $d$ of its children. Hence, $\ang{H_d}$ has only one element. Since $s-1$ is the greatest element with this property (given that it is the greatest minimal element of $\mathcal{P}$), $H_d = s-1$.

    \textbf{Case II:} $1 < k \le d < s$. In this case, adjacent elements of $\mathcal{E}$ that differ by $k$ are within $d$ of each other, so $\langle H_d \rangle$ can only have one minimal element. Since $s+k-1$ is the greatest element with this property (given that it is the greatest element of $\mathcal{E}$), $H_d = s+k-1$.

    \textbf{Cases III--VI:} $d < k$. By Definition~\ref{def-g},
    \[
        H_d = g(H_d) + (h(H_d)-1)s.
    \]
    In each case, we calculate $g(H_d)$ using Lemmas~\ref{lem-best-interval-ideal} and \ref{lem-gx}; we calculate $h(H_d)$ using Lemmas~\ref{lem-best-interval-ideal}, \ref{lem-hx}, \ref{lem-ledge-lengths}, and \ref{lem-esye-interval}.
\end{proof}

\section{Extension to the non-coprime case} \label{sec-non-coprime}

We begin by defining a variant of the notion of an order ideal generated by an element.

\begin{definition}
    Given $x \in \Z_{\ge 0}$, let
    \[
        \ang{x}_b = \{x-a_1bs-a_2b(s+k) \geq 0 \mid a_1,a_2 \in \Z_{\ge 0}\}.
    \]
\end{definition}

Notice that if $x \in \mathcal{P}_{bs, b(s+k)}$, then $\ang{x}_b$ is the order ideal generated by $x \in \mathcal{P}_{bs, b(s+k)}$. This notation gives us additional flexibility by allowing us to vary $b$ and allowing $x$ to not be an element of $\mathcal{P}_{bs, b(s+k)}$.

We first simplify the problem by proving that we may assume that $c=0$.

\begin{lemma} \label{lem-mult-of-b}
    We have $H_{bd+c}(bs,bk) = H_{bd}(bs,bk)$.
\end{lemma}
\begin{proof}
    Since $bd+c\geq bd$, we have $H_{bd+c}(bs,bk) \leq H_{bd}(bs,bk)$. Now, observe that all elements of $\ang{H_{bd}(bs,bk)}_b$ are congruent modulo $b$. Therefore, any two elements of $\ang{H_{bd}(bs,bk)}_b$ within $bd+c$ of each other are also within $bd$ of each other. Hence, $H_{bd}(bs,bk)$ is $(bd+c)$-distinct, so $H_{bd+c}(bs,bk) \geq H_{bd}(bs,bk)$.
\end{proof}

We now prove Theorem~\ref{thm-max-hook}, which says that for all integers $b \ge 2$ and $0\leq c < b$, we have
\[
    H_{bd+c}(bs,bk)=\begin{cases}
        b\left(H_{d}\left(s,k\right)+2\right)-1 &\quad\text{if} \ k=1 \ \text{and} \ d<s\\
        b\left(H_{d}\left(s,k\right)+1\right)-1 & \quad\text{if} \ k=1 \ \text{and} \ d\geq s \\
        b\left(H_{d}\left(s,k\right)+2\right)-1 &\begin{aligned}
            \quad\text{if}& \ d < k \ \text{and} \ (\sbar\esye\bmod k = 1 \\
            \quad &\text{or} \ d < \sbar \esye \bmod k = k-1)
        \end{aligned} \\
        b\left(H_{d}\left(s,k\right)+1\right)-1 &\begin{aligned}
            \quad\text{if}& \ k > 1 \ \text{and} \ (1 < \sbar \esye \bmod k \le d \\
            \quad &\text{or} \ (d < \sbar \esye \bmod k < k-1) \ \text{or} \ d\geq k).
        \end{aligned}
    \end{cases}
\]
\begin{proof}[Proof of Theorem~\ref{thm-max-hook}]
    By Lemma~\ref{lem-mult-of-b}, we may assume that $c=0$.
    
    \textbf{Case I:} $k=1 \ \text{and} \ d<s$. We have $H_d(s,1) = s-1$ by Theorem~\ref{thm-max-hook-coprime}. First, we prove $H_{bd}(bs,b) \ge b(s+1) - 1$. Since $b \nmid b(s+1) - 1$, we have $b(s+1) - 1 \in \mathcal{P}_{bs, b(s+1)}$. Further, $\ang{b(s+1) - 1}_b = \{bs+b-1, b-1\}$, which is $bd$-distinct.
    
    Now, we prove $H_{bd}(bs,b) \le b(s+1) - 1$. We have
    \[
        \ang{b(s+1)}_b=b\ang{s+1}_1=\{b(s+1), b, 0\},
    \]
    which is not $bd$-distinct. It follows that $\ang{x}_b$ is not $bd$-distinct for any $x \ge b(s+1)$.
    
    \textbf{Case II:} $k=1 \ \text{and} \ d\geq s$. We have $H_d(s,1) = s-1$ by Theorem~\ref{thm-max-hook-coprime}. First, we prove $H_{bd}(bs,b) \ge bs-1$. Since $b \nmid bs-1$, we have $bs - 1 \in \mathcal{P}_{bs, b(s+1)}$. Further, $\ang{bs-1}_b=\{bs-1\}$, which is $bd$-distinct.
    
    Now, we prove $H_{bd}(bs,b) \le bs - 1$. We have
    \[
        \ang{bs}_b=b\ang{s}_1 = \{bs,0\},
    \]
    which is not $bd$-distinct. It follows that $\ang{x}_b$ is not $bd$-distinct for any $x \ge bs$.
    
    \textbf{Case III:} $d < k \ \text{and} \ (\sbar\esye\bmod k = 1 \ \text{or} \ d < \sbar \esye \bmod k = k-1)$. First, we prove $H_{bd}(bs,bk)\geq b(H_d(s,k)+2)-1$. If $\sbar\esye\bmod k = 1$, then $s+k-1 \in \ang{H_d(s,k)}_1$ by Lemma~\ref{lem-best-interval-ideal}. And if $d < \sbar \esye \bmod k = k-1$, then $s-1 \in \ang{H_d(s,k)}_1$ by Lemma~\ref{lem-best-interval-ideal}. In particular,
    \[
        -1\in \{H_d(s,k)-a_1s-a_2(s+k) \mid a_1,a_2 \in \Z_{\ge 0}\}.
    \]
    Since $b(x+2)-1 \ge 0$ if and only if $x \ge -1$ for all $x \in \Z$, we have
    \[
        \ang{b(H_d(s,k)+2)-1}_b = b((\ang{H_d(s,k)}_1 \cup \{-1\}) + 2)-1.
    \]
    Thus, it suffices to prove that $\ang{H_d(s,k)}_1 \cup \{-1\}$ is $d$-distinct, which is equivalent to $[d-1] \cap \ang{H_d(s,k)}_1=\emptyset$. Suppose for the sake of contradiction that $x \in [d-1] \cap \ang{H_d(s,k)}_1$. Then, $x$ is the last element of $\mathcal{L}_x$. If $\sbar\esye\bmod k = 1$, then $\mathcal{L}_x$ is within $d$ of $\mathcal{L}_{k-1}$, which is impossible by Lemma~\ref{lem-best-interval-ideal}. If $d < \sbar \esye \bmod k = k-1$, then by Lemma~\ref{lem-best-interval-ideal}, $x+s \in \ang{H_d(s,k)}_1$, and $0 < |(x+s) - (s-1)| \le d$, contradicting the $d$-distinctness of $\ang{H_d(s,k)}_1$.
    
    Now, we prove $H_{bd}(bs,bk)\le b(H_d(s,k)+2)-1$. If $\sbar\esye\bmod k = 1$, then $s-2 \in \ang{H_d(s,k)}_1$ by Lemma~\ref{lem-best-interval-ideal}. And if $d < \sbar \esye \bmod k = k-1$, then $s+k-2 \in \ang{H_d(s,k)}_1$ by Lemma~\ref{lem-best-interval-ideal}. In particular,
    \[
        -2\in \{H_d(s,k)-a_1s-a_2(s+k) \mid a_1,a_2 \in \Z_{\ge 0}\}.
    \]
    Hence, we have
    \[
        \ang{b(H_d(s,k)+2)}_b = b((\ang{H_d(s,k)}_1 \cup \{-1,-2\}) + 2) \supseteq \{b,0\},
    \]
    which is not $bd$-distinct. It follows that $\ang{x}_b$ is not $bd$-distinct for any $x \ge b(H_d(s,k)+2)$.
    
    \textbf{Case IV:} $k > 1 \ \text{and} \ (1 < \sbar \esye \bmod k \le d \ \text{or} \ (d < \sbar \esye \bmod k < k-1) \ \text{or} \ d\geq k)$. First, we prove $H_{bd}(bs,bk) \ge b(H_d(s,k)+1)-1$. Since $b(x+1)-1 \ge 0$ if and only if $x \ge 0$ for all $x \in \Z$, we have
    \[
        \ang{b(H_d(s,k)+1)-1}_b = b(\ang{H_d(s,k)}_1 + 1)-1.
    \]
    Since $\ang{H_d(s,k)}_1$ is $d$-distinct, $\ang{b(H_d(s,k)+1)-1}_b$ is $bd$-distinct.
    
    Now, we prove $H_{bd}(bs,bk) \le b(H_d(s,k)+1)-1$. If $1 < \sbar \esye \bmod k \le d$ or $d \ge k,s$, then $s-1 \in \ang{H_d(s,k)}_1$ by Lemma~\ref{lem-best-interval-ideal} and Theorem~\ref{thm-max-hook-coprime}. And if $d < \sbar \esye \bmod k < k-1$ or $s > d \ge k$, then $s+k-1 \in \ang{H_d(s,k)}_1$ by Lemma~\ref{lem-best-interval-ideal} and Theorem~\ref{thm-max-hook-coprime}. In particular,
    \[
        -1\in \{H_d(s,k)-a_1s-a_2(s+k) \mid a_1,a_2 \in \Z_{\ge 0}\}.
    \]
    Hence, we have
    \[
        \ang{b(H_d(s,k)+1)}_b = b((\ang{H_d(s,k)}_1 \cup \{-1\}) + 1).
    \]
    Thus, it suffices to prove that $\ang{H_d(s,k)}_1 \cup \{-1\}$ is not $d$-distinct, which is equivalent to $[d-1] \cap \ang{H_d(s,k)}_1 \neq \emptyset$. If $1 < \sbar \esye \bmod k \le d$, then $(\sbar \esye \bmod k) - 1 \in [d-1] \cap \ang{H_d(s,k)}_1$ by Lemma~\ref{lem-best-interval-ideal}. If $d < \sbar \esye \bmod k < k-1$, then $k - (\sbar\esye \bmod k) - 1 \in [d-1] \cap \ang{H_d(s,k)}_1$ by Lemma~\ref{lem-best-interval-ideal}. If $d \ge k,s$, then $s-1 \in [d-1] \cap \ang{H_d(s,k)}_1$ by Theorem~\ref{thm-max-hook-coprime}. Finally, if $s > d \ge k$, then $k-1 \in [d-1] \cap \ang{H_d(s,k)}_1$ by Theorem~\ref{thm-max-hook-coprime}.
\end{proof}

\section*{Acknowledgments}
This project was partially supported by RTG grant NSF/DMS-1745638. It was supervised as part of the University of Minnesota School of Mathematics Summer 2022 REU program. We would like to thank our mentor Hannah Burson and our TA Robbie Angarone for their guidance and for carefully reading multiple drafts of this paper. We also thank Gregg Musiker for organizing this REU.

\bibliography{bibliography}{}
\bibliographystyle{amsplain}

\end{document}